\DeclareSymbolFont{AMSb}{U}{msb}{m}{n}
\documentclass[11pt,a4paper,twoside,leqno,noamsfonts]{amsart}
\usepackage[utopia]{mathdesign}
\usepackage[english]{babel}
\usepackage{amsmath, amsthm, amstext}
\usepackage{graphicx}
\usepackage{microtype}
\usepackage{url}
\usepackage[usenames]{color}
\usepackage[all]{xy}
\usepackage{enumerate}
\usepackage{comment}
\xyoption{all}
\usepackage{srcltx} 
\usepackage{breqn}

\newtheorem{thm}{Theorem}[section]
\newtheorem{lem}[thm]{Lemma}
\newtheorem{cor}[thm]{Corollary}
\newtheorem{prop}[thm]{Proposition}

\newtheorem{question}{Question}

\theoremstyle{definition}
\newtheorem{rem}[thm]{Remark}

\newtheorem{defn}[thm]{Definition}

\def\Adm{\mathcal{A}dm}
\def\bA{\mathbb{A}}
\def\bC{\mathbb{C}}

\def\cM{\mathcal{M}}

\def\cO{\mathcal{O}}
\def\bP{\mathbb{P}}

\def\bQ{\mathbb{Q}}

\def\bZ{\mathbb{Z}}

\DeclareMathOperator{\ch}{ch}

\DeclareMathOperator{\Id}{Id}

\DeclareMathOperator{\Proj}{Proj}
\DeclareMathOperator{\Qmod}{Qmod}

\DeclareMathOperator{\Spec}{Spec}

\DeclareMathOperator{\td}{td}

\begin{document}

\title{The Class of the $d$-elliptic Locus in Genus 2}

\author[C. Lian]{Carl Lian}
\address{Carl Lian: Department of Mathematics, Columbia University in the City of New York \hfill \newline\texttt{}
\indent 2990 Broadway,
New York, NY, 10027}
\email{{\tt clian@math.columbia.edu}}

\begin{abstract}
We compute the rational Chow class of the locus of genus 2 curves admitting a $d$-to-1 map to a genus 1 curve, recovering a result of Faber-Pagani when $d=2$. The answer exhibits quasi-modularity properties similar to those in the Gromov-Witten theory of a fixed genus 1 curve. Along the way, we give a classification of Harris-Mumford admissible covers of a genus 1 curve by a genus 2 curve, which may be of independent interest. As an application of the main calculation, we compute the number of $d$-elliptic curves in a very general family of genus 2 curves obtained by fixing five branch points of the hyperelliptic map and varying the sixth.
\end{abstract}
 
\maketitle

\section{Introduction}

The main question we consider is:

\begin{question}\label{main_question}
Let $F:X\to B$ be a family of genus 2 curves, where $B$ is a proper curve. How many fibers of $F$ are \textbf{$d$-elliptic}, i.e., admit a finite cover of an elliptic curve?
\end{question}

Question \ref{main_question} amounts to understanding the closure of the locus of $d$-elliptic curves in $\overline{\cM_2}$. In order to compute the stable limits of $d$-elliptic curves, we pass to the stack $\Adm_{2/1,d}$ of admissible covers, which compactifies the Hurwitz space of degree $d$ covers of a genus 1 curve by a genus 2 curve. We have a generically finite map $\pi_{2/1,d}:\Adm_{2/1,d}\to\overline{\cM_2}$ remembering the (stabilized) source of an admissible cover, and the image of $\pi_{2/1,d}$ compactifies the locus of $d$-elliptic curves.

Our main result is:

\begin{thm}\label{main_thm}
The class of the morphism $\pi_{2/1,d}$ in $A^{*}(\overline{\cM_2})$ is
\begin{equation*}         
\left(2\sigma_{3}(d)-2d\sigma_{1}(d)\right)\delta_0+\left(4\sigma_{3}(d)-4\sigma_{1}(d)\right)\delta_1,
\end{equation*}
where
\begin{equation*}
\sigma_k(d)=\sum_{d'|d}(d')^{k}.
\end{equation*}                      
\end{thm}

When $d=2$, Theorem \ref{main_thm} recovers \cite[Proposition 2]{faber}. Indeed, it follows from \S \ref{split_jac} that the morphism $\pi_{2/1,2}$ is generically 4-to-1, so the class of the substack of bielliptic curves is $\frac{1}{4}[\pi_{2/1,2}]=\frac{3}{2}\delta_0+6\delta_1$. We may also deduce:

\begin{cor}\label{modularity_statement}
We have
\begin{equation*}
\sum_{d}[\pi_{2/1,d}]q^d\in A^{1}(\overline{\cM_{2}})\otimes\Qmod,
\end{equation*}
that is, the generating functions of the coefficients of $\delta_0$ and $\delta_1$ lie the ring of quasi-modular forms.
\end{cor}

Analogous quasi-modularity statements hold in the all-genus Gromov-Witten theory of elliptic curves, see \cite[\S 5]{op}. Our results deal with the enumeration of covers of \textit{moving} elliptic curves, so it is natural to ask about generalizations of Theorem \ref{main_thm} and Corollary \ref{modularity_statement} to higher genus. See \S \ref{main_pf} for further discussion.


As an example application in the direction of Question \ref{main_question}, we prove:

\begin{thm}\label{enum_thm}
Let $x_1,\ldots,x_5\in\bP^1$ be a very general collection of points. Let $a_d$ be the number of points $x_6\in\bP^1$ such that the hyperelliptic curve branched over $x_1,\ldots,x_6$ is smooth and $d$-elliptic. Then,
\begin{equation*}
a_d=5d\left[\sum_{d'|d}\left(\frac{\sigma_{3}(d')}{d'}\cdot\mu\left(\frac{d}{d'}\right)\right)-d\right],
\end{equation*}
where $\mu(m)$ is the M\"{o}bius function.
\end{thm}


When $d=2$, Theorem \ref{enum_thm} says that there are 15 points $x_6$ such that the hyperelliptic curve branched over $x_1,\ldots,x_6$ is bielliptic, i.e., it admits a 2-to-1 cover of an elliptic curve. Indeed, there are $\frac{1}{2}\binom{5}{2}\binom{3}{2}=15$ ways to partition $x_1,\ldots,x_6$ into two pairs $\{a_1,a_2\},\{b_1,b_2\}$, and a fifth point $c$. When $x_1,\ldots,x_5$ are general, there is a unique $x_6\in\bP^1$ such that there is an involution of $\bP^1$ swapping $a_1$ with $a_2$, $b_1$ with $b_2$, and $c$ with $x_6$. The quotient by this involution has genus 1.

The structure of the paper is as follows. We collect some preliminary results in \S \ref{prelim}. We give a complete classification of admissible covers of a genus 1 curve by a genus 2 curve in \S \ref{admissible_classification}: while not all covers that appear will arise in our main computation, the classification may be of independent interest. In \S \ref{tE_intersection} and \S \ref{eta_E_intersection}, we intersect $\pi_{2/1,d}:\Adm_{2/1,d}\to \overline{\cM_2}$ with two test curves in the boundary: $t_E$, formed by gluing a fixed general elliptic curve to a varying elliptic curve along their origins, and $\eta_E$, formed by gluing a variable point to the origin of a fixed elliptic curve. Given the classification of \S \ref{admissible_classification}, understanding the set-theoretic intersections is straightforward, but the intersections are not transverse: we obtain the intersection multiplicities by carrying out explicit computations in the deformation spaces of admissible covers. From here, we easily deduce Theorem \ref{main_thm} and Corollary \ref{modularity_statement} in \S \ref{main_pf}. Finally, we prove Theorem \ref{enum_thm} in \S \ref{application}.

\subsection{Acknowledgments}
I am grateful to my advisor Johan de Jong, who suggested the initial research directions and offered invaluable guidance throughout. I also thank Jim Bryan, Raymond Cheng, Bong Lian, Georg Oberdieck, and Nicola Pagani for helpful comments. This work was completed with the support of an NSF Graduate Research Fellowship.

\section{Preliminaries}\label{prelim}

\subsection{Conventions}\label{conventions}

We work over $\bC$. Fiber products are over $\Spec(\bC)$ unless otherwise stated. All curves, unless otherwise stated, are assumed projective and connected with only nodes as singularities. The \textit{genus} of a curve $X$ refers to its arithmetic genus and is denoted $p_a(X)$. A \textit{rational} curve is an irreducible curve of geometric genus 0, and a \textit{nodal cubic} is a rational curve with one node. All moduli spaces are understood to be moduli stacks, rather than coarse spaces. 

All Chow rings are taken with rational coefficients and are denoted $A^{*}(X)$, where $X$ is a variety or Deligne-Mumford stack over $\bC$. We will frequently refer to the Chow class of a morphism $f:Y\to X$, by which we mean the bivariant class (see \cite[Chapter 17]{fulton}). However, $f$ will always be generically finite and $Y,X$ will always be proper, so we may as well take the proper pushforward $f_{*}([Y])$. The class of $f$ in $A^{*}(X)$ is denoted $[f]$.

\subsection{The Chow Ring of $\overline{\cM_2}$}\label{chow_ring_m2}

In this section, we recall Mumford's computation of $A^{*}(\overline{\cM_2})$ in \cite{mumford}. Let us first define the classes that appear.

Let $\Delta_{0}$ be the closure of the locus in $\overline{\cM_2}$ of irreducible nodal curves, or alternatively the image of the generically 2-to-1 morphism $b_{0}:\overline{\cM_{1,2}}\to\overline{\cM_2}$ gluing the two marked points of a curve in the source. Let $\Delta_{1}$ be the closed substack of reducible curves, or alternatively the image of the 2-to-1 morphism $b_{1}:\overline{\cM_{1,1}}\times\overline{\cM_{1,1}}\to\overline{\cM_2}$ gluing the marked points. Let $\delta_i$ denote the class of $\Delta_i$ in $A^{1}(\overline{\cM_2})$, by which we always mean the Q-class.

Let $\pi:\mathcal{C}_2\to\overline{\cM_2}$ be the universal family, and define 
\begin{equation*}
\lambda_1=c_1(\Lambda^2\pi_{*}\omega_{\mathcal{C}_2/\overline{\cM_2}}).
\end{equation*}

Let $\Delta_{00}$ be the closure of the locus of rational curves with two nodes, or alternatively the image of the 8-to-1 morphism $b_{00}:\overline{M_{0,4}}\to\overline{\cM_2}$ gluing together the first two and last two marked points. Let $\Delta_{01}$ be the locus of reducible genus 2 curves having a nodal cubic as a component, or alternatively the image of the generically 2-to-1 morphism $b_{01}:\overline{\cM_{1,1}}\to\overline{\cM_2}$ that attaches a nodal cubic at the marked point. Let $\delta_{0i}$ denote the (Q-)class of $\Delta_{0i}$ in $A^{1}(\overline{\cM_2})$.

We also let $C_{000}$ denote the stable curve obtained by gluing two $\bP^1$s at three points, and let $C_{001}$ the stable curve obtained by attaching two nodal cubics at smooth points.

Let $p$ denote the bivariant class of a geometric point $\Spec(\bC)\to\overline{\cM_2}$; all such points are rationally equivalent.

\begin{thm}\cite[Theorem 10.1]{mumford}\label{M2_relations}
The classes $1,\delta_0,\delta_1,\delta_{00},\delta_{01},p$ form a $\bQ$-basis of $A^{*}(\overline{\cM_2})$. Moreover, we have the following relations\footnote{There appears to be a sign error in the second to last relation stated in \cite{mumford}; it is stated correctly in the second list of relations Mumford gives in terms of $\lambda_1$.}:
\begin{align*}
10\lambda_1&=\delta_0+2\delta_1\\
\delta_0^{2}&=\frac{5}{3}\delta_{00}-2\delta_{01}\\
\delta_0\delta_1&=\delta_{01}\\
\delta_1^{2}&=-\frac{1}{12}\delta_{01}\\
\delta_0\cdot\delta_{00}&=-\frac{1}{4}p\\
\delta_{0}\cdot\delta_{01}&=\frac{1}{4}p\\
\delta_{1}\cdot\delta_{00}&=\frac{1}{8}p\\
\delta_{1}\cdot\delta_{01}&=-\frac{1}{48}p.
\end{align*}
\end{thm}

\subsection{Admissible covers}\label{admissible_covers_prelim}
We recall the definition of \cite{hm}:

\begin{defn}
Let $X,Y$ be curves. Let $b=(2p_a(X)-2)-d(2p_a(Y)-2)$, and let $y_1,\ldots,y_b\in Y$ be distinct marked points so that $(Y,y_1,\ldots,y_b)$ is stable in the sense of \cite{dm}. Then, an \textbf{admissible cover} consists of the data of the stable marked curve $(Y,y_1,\ldots,y_b)$ and a finite morphism $f:X\to Y$ such that:
\begin{itemize}
\item $f(x)$ is a smooth point of $Y$ if and only if $x$ is a smooth point of $X$,
\item $f$ is simply branched over the $y_i$ and \'{e}tale over the rest of the smooth locus of $Y$, and
\item at each node of $X$, the ramification indices of $f$ restricted to the two branches are equal.
\end{itemize}
\end{defn}

\begin{rem}\label{preimage_adm_smooth}
It is clear that non-separating nodes of $X$ must map to non-separating nodes of $Y$. Hence, the preimage of a smooth component of $Y$ must be a union of smooth components of $X$. 
\end{rem}

Admissible covers of degree $d$ from a genus $g$ curve to a genus $h$ curve are parametrized by a proper Deligne-Mumford stack $\Adm_{g/h,d}$, see \cite{hm,mochizuki,acv}. Let $b=2g-2-d(2h-2)$ be the number of branch points of such a cover. Then, we have a forgetful map $\psi_{g/h,d}:\Adm_{g/h,d}\to\overline{\cM_{h,b}}$ remembering the target, and another $\pi_{2/1,d}:\Adm_{g/h,d}\to\overline{\cM_g}$ remembering the source with unstable rational components contracted, as in \cite{knudsen}.

We also recall from \cite{hm} the explicit local description of $\Adm_{g/h,d}$. Let $f:X\to Y$ be a point of $\Adm_{g/h,d}$. Let $y'_1,\ldots,y'_n$ be the nodes of $Y$, and let $y_{1},\ldots,y_{b}\in Y$ be the (smooth) branch points. Let $\mathbb{C}[[t_1,\ldots,t_{3h-3+b}]]$ be the versal deformation space of $(Y,y_1,\ldots,y_b)$, so that $t_1,\ldots,t_n$ are smoothing parameters for the nodes $y'_1,\ldots,y'_n$. Let $x_{i,1},\ldots,x_{i,r_i}$ be the nodes of $X$ mapping to $y'_i$, and denote the ramification index of $f$ at $x_{i,j}$ by $a_{i,j}$. 

\begin{prop}[\cite{hm}]\label{adm_defos}
The complete local ring of $\Adm_{g/h,d}$ at $[f]$ is
\begin{equation*}
\mathbb{C}[[t_1,\ldots,t_{3h-3+b},\{t_{i,j}\}^{1\le i\le n}_{1\le j\le r_i}]]/(t_1=t_{1,1}^{a_{1,1}}=\cdots=t_{1,r_1}^{a_{1,r_1}},\ldots,t_n=t_{n,1}^{a_{n,1}}=\cdots=t_{n,r_n}^{a_{n,r_n}}).
\end{equation*} 
\end{prop}

Here, the variable $t_{i,j}$ is the smoothing parameter for $X$ at $x_{i,j}$. In particular, we see that $\Adm_{g/h,d}$ is Cohen-Macaulay of pure dimension $3h-3+b$. Moreover, If the $a_{i,j}$ are all equal to 1, that is, $f$ is unramified over the nodes of $Y$, then $\Adm_{g/h,d}$ is smooth at $[f]$.

\subsection{Counting branched covers of curves}

The following lemmas are standard.

\begin{lem}\label{count_isogenies}
Let $E$ be a fixed elliptic curve and $d$ a positive integer. Then, the number of isomorphism classes of isogenies $E\to F$ of degree $d$ is $\sigma_1(d)$. Likewise, the number of isomorphism classes of isogenies $F\to E$ of degree $d$ is $\sigma_1(d)$.
\end{lem}

\begin{proof}
We see that these two numbers are equal by taking duals, so it suffices to count isogenies $E\to F$ of degree $d$, i.e., quotients of $E$ by a subgroup of order $d$, which is the number of index $d$ sublattices of $\bZ^2$. A sublattice of $\bZ^2$ is determined by a $\bZ$-basis $(a,0),(b,c)$, where $a,c$ are positive and as small as possible; $b$ is uniquely determined modulo $a$. As $ac=d$, the number of such sublattices is exactly $\sigma(d)$.
\end{proof}

\begin{lem}\label{count_covers_from_elliptic_curve}
Let $(E,0)$ be an elliptic curve and $d$ a positive integer. Then, up to automorphisms of the target, there are $d^2-1$ morphisms $f:E\to\bP^1$ totally ramified at 0 and at some other point $x\in E$, one for each $x\in E[d]-\{0\}$. When $E$ is general, $f$ is simply branched over two distinct points of $\bP^1$.
\end{lem}
\begin{proof}
The linear system defining $f$ must be a 2-dimensional subspace $W$ of $V=H^0(E,\cO(d\cdot0))$. In order for $W$ to be totally ramified at $x$, we need $\cO(d\cdot0)\cong\cO(d\cdot x)$, that is, $x\in E[d]-\{0\}$. For such an $x$, there are unique (up to scaling) sections in $V$ vanishing to maximal order at 0 and $x$; thus $W$ is uniquely determined. Moreover, $f$ will be simply branched over two distinct points of $\bP^1$ unless it has two simple ramification points over the same point of $\bP^1$ or a triple ramification point; however, this will only happen for finitely many $E$, exhibited as a 3-point cover of $\bP^1$.
\end{proof}

\subsection{Genus 2 Curves with Split Jacobian}\label{split_jac}

In this section, all curves are assumed to be smooth, and $J(X)$ denotes the Jacobian of $X$. The main reference here is \cite[\S 2]{kuhn}.

\begin{defn}
Let $f:X\to E$ be a morphism of curves of degree $d$, where $C$ has genus 2 and $E$ has genus 1. We say that $f$ is \textit{optimal} if it does not factor as $X\to E'\to E$, where $E'\to E$ is an isogeny of degree greater than 1.
\end{defn}

Let $f:X\to E_1$ be an optimal cover of degree $d$. Fix a Weierstrass point $x_0\in X$, and let $\iota:X\mapsto J(X)$ be the embedding sending $x\mapsto \cO(x-x_0)$. We may regard $E_1$ as an elliptic curve with origin $f(x_0)$. We then get an induced morphism of abelian varieties $\phi_1:J(X)\to E_1$ such that $\phi_1\circ\iota=f$. We have an exact sequence
\begin{equation}\label{jacobian_ses}
\xymatrix{
0 \ar[r] & E_2 \ar[r] & J(X) \ar[r] & E_1  \ar[r] &  0
}
\end{equation}
By the optimality of $f$, $E_2$ is connected. Let $\phi_2:J(X)\to E_2$ be the dual map to the embedding $\overline{\phi_2}:E_2\mapsto J(X)$, and let $f_2=\phi_2\circ \iota:X\to E_2$.

\begin{lem}[\cite{kuhn}, \S 2]\label{kuhn}
$f_2$, as constructed above is an optimal cover of degree $d$, and $\phi=\phi_1\oplus\phi_2:J(X)\to E_1\oplus E_2$ is an isogeny of degree $d^2$.
\end{lem}

%
%

\begin{cor}\label{complement_not_isogenous}
Let $X$ be a $d$-elliptic curve of genus 2, where $d$ is minimal, and let $f:X\to E_1$ be an optimal cover of degree $d$. Let $f_2:X\to E_2$ and $\phi:J(X)\to E_1\oplus E_2$ be as above. Suppose that $E_1$ and $E_2$ are not isogenous. Then, any non-constant morphism $f_0:X\to E_0$, where $E_0$ is a curve of genus 1, factors uniquely through exactly one of $f_1$ and $f_2$.
\end{cor}

\begin{proof}
We may regard $E_0$ as an elliptic curve with origin $f_0(x_0)$. We then get an induced morphism of abelian varieties $\phi_0:J(X)\to E_0$ such that $f=\phi_0\circ \iota$, and a nonzero dual morphism $\overline{\phi_0}:E_0\to J(X)$. Exactly one of the maps $\phi'_i:=\mathrm{pr}_i\circ \phi\circ \overline{\phi_0}$ must be nonzero, because $E_1$ and $E_2$ are not isogenous; assume that $\phi'_1=0$ and $\phi'_2\neq0$. Then, from the exact sequence $(\ref{jacobian_ses})$, we have that $\overline{\phi_0}$ factors as $\overline{\phi_2}\circ g$, for some non-zero $g:E_0\to E_2$. Dualizing and pre-composing with $\iota$ shows that $f_0$ factors through $f_2$. The uniqueness of $f_0$ follows from the uniqueness of the factorization $\overline{\phi_0}=\overline{\phi_2}\circ g$.

\end{proof}

\section{Classification of Admissible Covers}\label{admissible_classification}

In this section, we classify the points of $\Adm_{2/1,d}$. Let $f:X\to Y$ be such a point, and let $y_1,y_2\in Y$ denote the marked branch points. The classification proceeds by pulling back the strata usual boundary strata of $\overline{\cM_{1,2}}$ by the map $\psi_{2/1,2}:\Adm_{2/1,d}\to\overline{\cM_{1,2}}$, and studying the admissible covers where the target curve is of each possible topological type.

The base curve $(Y,y_1,y_2)$ is of one of the following types:
\begin{enumerate}
\item[(1)] a smooth curve of genus 1 with two marked points,
\item[(2)] a smooth curve of genus 1 with a $\bP^1$ attached at one point, with both marked points on the rational component,
\item[(3)] a nodal cubic with two marked points,
\item[(4)] a nodal cubic with a $\bP^1$ attached at one point, with both marked points on the smooth component, or
\item[(5)] the union of two $\bP^1$s attached at two points, with one marked point on each component.
\end{enumerate}

Covers of type 1 are ``generic'': they lie in the open locus (Hurwitz space) $\Adm_{2/1,d}^{\circ}$ of covers of smooth curves. Covers of types 2 and 3 may be regarded as boundary divisors (curves) on $\Adm_{2/1,d}$, and are the most important for us when we restrict $\pi_{2/1,d}$ to test families in the boundary of $\overline{\cM_{2}}$. Covers of types 4 and 5 are specializations of covers of types 2 and 3, and will not appear in our test families.

We now describe all admissible covers $f:X\to Y$ in each of these five cases. 

\subsection{Type 1}\label{base_smooth}

If $Y$ is smooth, then $X$ must be smooth, see Remark \ref{preimage_adm_smooth}.

\subsection{Type 2}

Let $Y_i$ be the component of $Y$ of genus $i$, and let $y=Y_0\cap Y_1$. We classify admissible covers of $Y$ into three sub-types according to the topology of $f^{-1}(Y_1)$, which is a disjoint union of smooth curves. There are three possibilities: $f^{-1}(Y_1)$ may consist of a single genus 2 curve (type 2A/2A$'$), a single genus 1 curve (type 2B), or two genus 1 curves (type 2C). The resulting covers are shown in Figure \ref{fig:type2cover}, with details given below.

\begin{figure}
  \caption{Covers of types 2A, 2A$'$, 2B, 2C.}
  \centering
    \includegraphics[width=\textwidth]{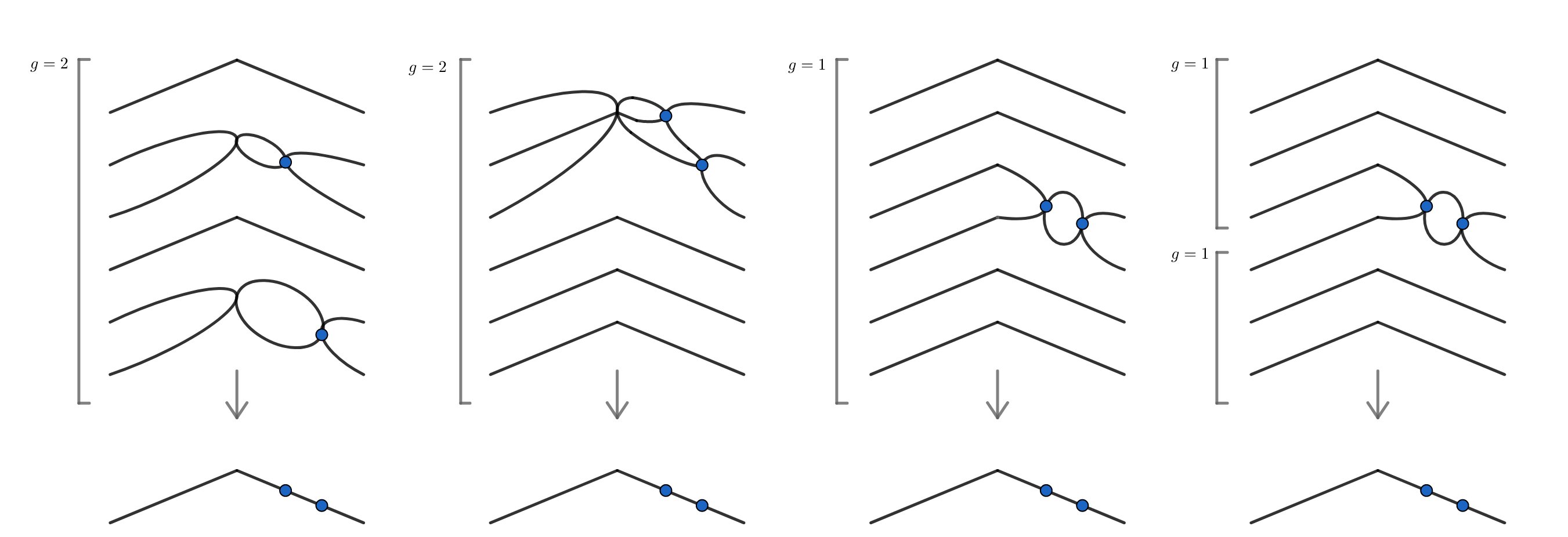}
    \label{fig:type2cover}
\end{figure}

\subsubsection{Type 2A/2A$'$}
Suppose $X$ has a genus 2 component $X_2$ mapping to $Y_1$. Then, the branch divisor of $f|_{X_2}$ must be equal to $2y$, and hence $f^{-1}(y)\subset X_2$ contains two points $x',x''\in X_2$ with $e_{x'}=e_{x''}=2$ (type 2A), or a unique point $x\in X_2$ with $e_{x'}=3$ (type 2A$'$). Then, at each point $x\in f^{-1}(y)$, we must attach a smooth rational component $R_x$ mapping to $Y_1$ with degree at least $e_x$. The $R_x$ must be pairwise disjoint, and can only contain two additional ramification points in total, so in fact the degree of $R_x$ over $Y_0$ is equal to $e_x$.

In both cases, we have $\pi_{2/1,d}([f])=[X_2]\in\cM_2$.

\subsubsection{Type 2B}
Suppose $X$ has a single genus 1 component $X_1$ mapping to $Y_1$. The map $f|_{X_1}:X_1\to Y_1$ must be unramified. The ramification divisor of $f$ has degree 2, so $f^{-1}(Y_0)$ must be a union of rational curves, attached to the points of $X_1$ lying over $y$. Exactly one pair of these points lie on a common rational component $X'$ of degree 2 over $Y_0$; all other components of $f^{-1}(Y_0)$ map isomorphically to $Y_0$.

In this case, $\pi_{2/1,d}([f])$ is the curve obtained by gluing two points of $X_1$ together. Note that the glued points both lie in the kernel of an isogeny of degree $d$, so their difference must be $d$-torsion.

\subsubsection{Type 2C}
Suppose $X$ has two genus 1 components $X_1,X'_1$ mapping to $Y_1$. First, note that $X_1$ and $X'_1$ must be disjoint, or else they would intersect at a node $x$ lying over $y$, but neither branch at $x$ would map to the branch at $y$ on $Y_0$. As before in type 2B, $f$ must be unramified over $Y_1$, and $f^{-1}(Y_0)$ must be a union of rational curves, attached to the points of $X_1$ lying over $y$. One of these rational curves must connect $X_1$ to $X'_1$ in order for $X$ to be connected, and its degree over $Y_0$ is 2; all other rational components of $X$ map isomorphically to $Y_0$, and the rational components of $X$ are pairwise disjoint.

In this case, $\pi_{2/1,d}([f])=[X_1\cup X'_1]$, where the two components are attached at a node.

\subsection{Type 3}

Throughout this discussion, we fix the normalization map $\nu:\bP^1\to Y$, and let $\widetilde{f}:\widetilde{X}\to \bP^1$ be the map obtained by normalizing $X$ and $Y$. Let $y_0\in Y$ denote the node, and let $y',y''\in\bP^1$ be its pre-images under $\nu$. By abuse of notation, we let $y_1,y_2\in\bP^1$ denote the pre-images of $y_1,y_2\in Y$. Then, $\widetilde{f}:\widetilde{X}\to \bP^1$ is simply branched over $y_1,y_2$, possibly branched over $y',y''$, and unramified everywhere else.

Let $\widetilde{X}_1,\ldots,\widetilde{X}_n$ be the components of $\widetilde{X}$, and let $d_i$ be the degree of $\widetilde{X}_i$ over $\bP^1$. Let $s_i$ denote the total number of points of $\widetilde{X}_i$ lying over $y'$ and $y''$. Then, the number of points of $f^{-1}(y)$ is $t=\frac{1}{2}\sum s_i$. We have $p_a(\widetilde{X})\ge1-n$. On the other hand, $\widetilde{X}$ is is the blowup of $X$ at $t$ points, so $p_a(\widetilde{X})=2-t$. Thus $2-t\ge 1-n$, hence $t\le n+1$. On the other hand, $s_i\ge2$ for each $i$, so $t\ge n$.

Because $t$ is an integer, the only possibilities for the $s_i$ (up to re-indexing) are: $s_1=4$ and $s_i=2$ for all $i\ge2$ (type 3A), $s_1=s_2=3$ and $s_i=2$ for all $i\ge3$ (type 3B/3B$'$), or $s_i=2$ for all $i$ (type 3C). The possibilities for the source curve $X$, are shown in Figure \ref{fig:type3cover}, with justification below.

\begin{figure}
  \caption{Sources for covers of types 3A, 3B, 3B$'$, 3C. Sources for covers of types 4A$'$, 4B, 4B$'$, 4C$'$ may be obtained by attaching rational curves. Sources for covers of types 4A, 4C by additionally replacing the two smooth ramification points with a smooth triple ramification point.}
  \centering
    \includegraphics[width=\textwidth]{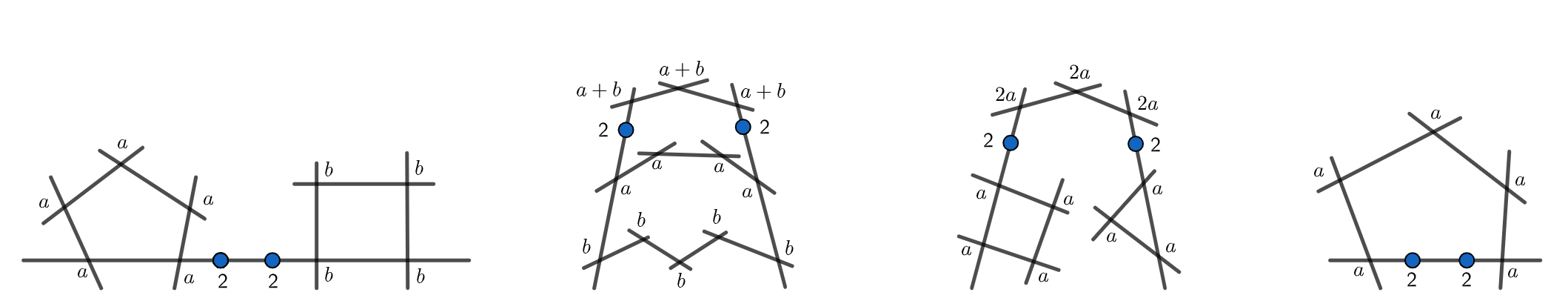}
    \label{fig:type3cover}
\end{figure}

\subsubsection{Type 3A}
We have a single component $X_0\subset X$ with four points mapping to $y_0\in Y$; all other components of $X$ have two points mapping to $y_0$. As $X_0$ is connected, we see that $X$ must consist of two disjoint chains of curves $X_1,\ldots,X_p$ and $X'_1,\ldots,X'_q$ attached at two points to $X_0$. We allow possibility that at least one of $p,q$ is equal to 0; in this case, $X_0$ has a non-separating node. We deal here with the case $p,q>0$, and remark at the end that the ``degenerate'' cases are similar. With this assumption, we must have that all components of $X$ have genus 0.

Each of the components of $X$ other than $X_0$ is unramified over $\bP^1$ away from the two nodes; thus, each $X_i\to\bP^1$ is of the form $x\mapsto x^a$ for some $a$ (independent of $i$), branched over $y',y''$. Similarly each $X'_j\to\bP^1$ is of the form $x\mapsto x^b$ for some $b$ (independent of $j$), branched over $y',y''$.

Now, $X_0$ has degree $a+b$ over $\bP^1$, and each of $y',y''$ has two points in its pre-image, of ramification indices $a$ and $b$. (Note that when $a=b$, we have three ways of distributing the nodes on $X_0$ over $y'$ and $y''$). By Riemann-Hurwitz, there are two additional simple ramification points on $X_0$ mapping to $y_1,y_2\in Y$.

The situation is similar when at least one of $p,q$ is equal to 0: the chains of smooth rational curves attached to $X_0$ may be replaced with a non-separating node on $X_0$, and the normalization of $X_0$ maps to $\bP^1$ as before.

In this case, $\pi_{2/1,d}([f])\in\Delta_{00}-[C_{001}]$.

\subsubsection{Type 3B/3B$'$}
We have two components $X_1,X_2$ containing three points each mapping to $y_0\in Y$. $X_1$ and $X_2$ are connected either by three possibly empty chains of curves (type 3B), or one chain of curves, with chains of rational curves attached to each $X_i$ at two points (type 3B$'$). As in type 3A, the latter chains may be empty, in which case at least one of $X_1$ or $X_2$ may become nodal. We first assume this is not the case.

Then, as before, every component of $X$ must have genus 0, and all components of $X$ other than $X_1$ and $X_2$ map to to $\bP^1$ via $x\mapsto x^r$, for some $r\ge1$. The three nodes on each $X_i$ are the only points mapping to $y_0$, so $f|_{X_i}$ must be totally ramified at one of them, and the sum of the ramification indices at the other two must be equal to the degree of $f|_{X_i}$. By Riemann-Hurwitz, there is one additional simple ramification point on each $X_i$.

If, in type 3B', two nodes of $X_1$ or $X_2$ are connected by an empty chain of curves (that is, glued to each other), then we obtain similar covers: $X_i$ becomes a nodal cubic, and its normalization maps to $Y$ as above.

In this case, we have $\pi_{2/1,d}([f])=[C_{000}]$, and in type 3B$'$, we have $\pi_{2/1,d}([f])=[C_{001}]$.

\subsubsection{Type 3C}

Here, $X$ must consist a smooth genus 1 component $X_1$ attached at two points to a chain of $p$ rational curves, each of which maps to $Y_0$ via $x\mapsto x^a$. Here, we need $a\ge2$, as $X_1$ will map to $\bP^1$ of $Y$ with degree $a$. The map $f:X_1\to\bP^1$ is totally ramified at the two nodes on $X_1$, and simply ramified at two other points on $X_1$.

As will be important for us later, we may also have $p=0$, in which case $X$ is irreducible with a single node, and its normalization $X_1$ maps to $\bP^1$ as above.

In this case, $\pi_{2/1,d}([f])$ is the nodal curve obtained by gluing two $a$-torsion points of $X_1$ together.
\subsection{Type 4}

As in type 2, let $Y_i$ be the component of $Y$ of genus $i$, and let $y=Y_0\cap Y_1$. By studying the restriction of $f$ over $Y_1$ in the same way as in type 3, we have the following possibilities:
\begin{itemize}
\item $f^{-1}(Y_1)$ is isomorphic to a source curve $X$ from types 3A, 3B/3B', or 3C, disregarding for now the smooth ramification points. (types 4A/4A', 4B/4B', 4C/4C')
\item $f^{-1}(Y_1)$ is isomorphic to a cycle of $p$ smooth rational curves, for $p\ge1$ (type 4D/4D$'$). When $p=1$, $f^{-1}(Y_1)$ is a nodal cubic.
\item $f^{-1}(Y_1)$ is isomorphic to two disjoint cycles of smooth rational curves of lengths $p,q$ (type 4E). As above, we allow $p=1$ and $q=1$.
\end{itemize}

We will see that in the first case, the covers of type 4 may essentially be constructed from those of type 3, by allowing the smooth ramification points to come together. However, the covers of types 4D, 4D$'$, and 4E are new: their source curves are shown in are shown in Figure \ref{fig:type4cover}.

\begin{figure}
  \caption{Sources for covers of types 4D, 4D$'$, 4E}
  \centering
    \includegraphics[width=\textwidth]{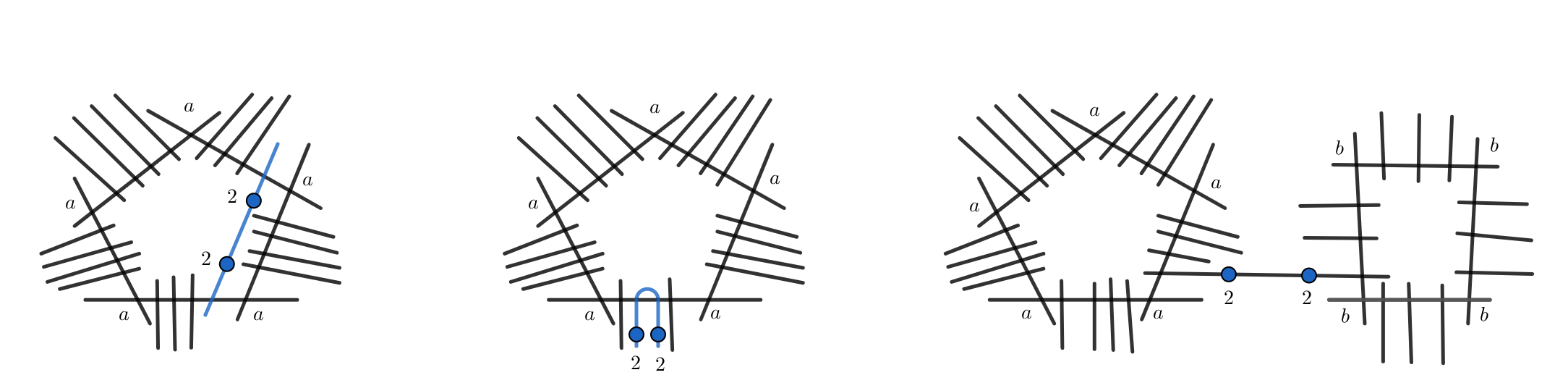}
    \label{fig:type4cover}
\end{figure}

\subsubsection{Type 4A/4A$'$}

Here, $f^{-1}(Y_1)$ is isomorphic to the source curve $X$ in type 3A. However, $f$ must now be unramified away from the nodes on $X_0$, so either $X_0$ acquires a triple ramification point over $y$ (type 4A), or its two additional ramification points must both lie over $y$ (type 4A$'$). In addition, at each $x\in f^{-1}(y)$, we must attach a smooth rational component to $f^{-1}(Y_1)$ mapping to $Y_0$ with degree $e_x$. In both cases, $\pi_{2/1,d}([f])\in\Delta_{00}-[C_{001}]$, as in type 3A.

\subsubsection{Type 4B/4B$'$}

Here, $f^{-1}(Y_1)$ is isomorphic to a source curve $X$ in types 3B/3B$'$. The simple ramification points on $X_1,X_2$ are now constrained both to lie over $y$, and as in type 4A/4A$'$, at each $x\in f^{-1}(y)$, we attach a smooth rational component to $f^{-1}(Y_1)$ mapping to $Y_0$ with degree $e_x$. The contractions of $X$ are as in types 4B/4B$'$.

\subsubsection{Type 4C/4C$'$}

Here, $f^{-1}(Y_1)$ is isomorphic to a source curve $X$ in type 3C. As in the previous two cases, we attach smooth rational curves to the pre-images of $x$ on $f^{-1}(Y_1)$, mapping to $Y_0$ with degree $e_x$. In addition, the two additional ramification points on the genus 1 component $X_1$ either come together to a triple ramification point over $y$ (type 4C), or are constrained both to lie over $y$ (type 4C$'$). 

As in type 3C, $X$ contracts to an irreducible genus 2 curve $X^c$ with one node. The normalization $X_1$ of $X^c$ is an elliptic curve with a cover of $\bP^1$ branched over only three points, so if we fix the degree of our admissible cover $f$, there are only finitely many possible isomorphism classes of $X_1$.

\subsubsection{Type 4D/4D$'$}

Suppose now that $f^{-1}(Y_1)$ is a cycle of smooth rational curves (or a nodal cubic). The restriction of $f$ to these rational components will be unramified over $y$, and we will again attach smooth curves at the points of $f^{-1}(y)$, which we find must be rational and disjoint. However, to make $X$ have the correct genus, two of these rational components must \textit{coincide}, i.e., we have a rational bridge between two distinct components of $f^{-1}(Y_1)$ (type 4D) or a rational bridge between two points on the same component of $f^{-1}(Y_1)$ (type 4D$'$). In both cases, the bridge must map to $Y_0$ with degree 2, branched over $y_1,y_2\in Y_0$.

In type 4D, we have $\pi_{2/1,d}{[f]}=[C_{000}]$, and in type 4D$'$, we have $\pi_{2/1,d}([X])\in\Delta_{00}-[C_{001}]$.

\subsubsection{Type 4E}

We argue as in the previous cases. If $p,q>1$, then the components of $f^{-1}(Y_1)$ must have genus 0, and map to $Y_1$ via $x\mapsto x^r$ for some $r\ge1$. We attach components of genus 0 to the pre-images of $y$, and in order to make $X$ connected, one such component is a bridge between the two cycles of $f^{-1}(Y_1)$ mapping to $Y_0$ with degree 2; the rest map isomorphically. As before, if instead $p=1$ or $q=1$, the corresponding cycle of smooth rational curves is replaced by a single nodal cubic. We find that $\pi_{2/1,d}{[X]}=[C_{001}]$.

%

\subsection{Type 5}

Let $Y=Y_1\cup Y_2$, with $Y_1\cap Y_2=\{y,y'\}$ and $y_i\in Y_i$. We first show that every component of $X$ has genus 0. Indeed, suppose that some smooth component $X_1$ of $X$ maps to $Y_1$. If the degree of $f|_{X_1}:X_1\to Y_1$ is $d_1$, then its branch divisor has degree at most $2d_1-1$, because the coefficients of the branch divisor at $y_1,y,y'$ are at most $1,d_1-1,d_1-1$, respectively, and zero everywhere else. Thus, by Riemann-Hurwitz, $X_1$ must be rational and coefficients of the branch divisor must be either $1,d_1-2,d_1-1$ or $1,d_1-1,d_1-2$. Suppose we are in the first case: then, $y$ has two preimages $x,\overline{x}\in X_1$ such that $e_{x}+e_{\overline{x}}=d_1$, and $y'$ has a single preimage $x'\in X_1$ with $e_{x'}=d_1$.

Let $n$ be the total number of components of $X$. Then, there are $3n/2$ points in total lying over $y$ and $y'$. Blowing up at these $3n/2$ points leaves $n$ disjoint curves of genus 0, and hence $2p_a(X)=-n+1+3n/2=n/2+1$. Thus $n=2$, and $X\cong C_{000}$; the ramification indices at the three nodes of $X$ must be $a,b,d/2$, where $a+b=d/2$, and there must be one smooth simple ramification point on each component of $X$.

\section{Intersection with the first test family}\label{tE_intersection}

We first define the test family, $t_E$, which we intersect with $\pi_{2/1,d}$. We fix, throughout this section, a general elliptic curve $(E,0)$. Let $t_E:\overline{\cM_{1,1}}\to\overline{\cM_2}$ be the morphism of stacks sending $(E',0')$ to the genus 2 curve $E\cup E'/(0\sim 0')$.

\begin{lem}\label{compare_delta01_to_map_from_M11}
We have $[t_E]=2\delta_{01}$ in $A^{2}(\overline{\cM_2})$.
\end{lem}

\begin{proof}
The classes of $[E]$ and a nodal cubic are rationally equivalent in $\overline{\cM_{1,1}}$. Pushing forward by the finite morphism $\overline{\cM_{1,1}}\times\overline{\cM_{1,1}}\to\overline{\cM_2}$, the classes of $b_{01}:\overline{\cM_{1,1}}\to\overline{\cM_2}$ (see \S \ref{chow_ring_m2}) and $t_E$ are equal. Because $b_{01}$ has degree 2 over its image $\Delta_{01}$, we have $[t_E]=2\delta_{01}$.
\end{proof}

We now consider the intersection of $t_E$ and $\pi_{2/1,d}$. Because $\Adm_{2/1,d}$ is Cohen-Macaulay and $\overline{\cM_2}$ is smooth, intersection products with $\pi_{2/1,d}$ are fiber products, as long as the intersections are proper, i.e., in the expected dimension.

\begin{lem}\label{groupoid_int_with_delta01}
The groupoid $(\overline{\cM_{1,1}}\times_{\overline{\cM_2}}\Adm_{2/1,d})(\Spec(\bC))$ is a set, i.e., its objects have no non-trivial automorphisms. (Here, the fiber product is of the morphisms $t_E$ and $\pi_{2/1,d}$.)
\end{lem}

\begin{proof}
An object of this groupoid consists of the data of an elliptic curve $(E',0')$, an admissible cover $f:X\to Y$, and an isomorphism of $t_{E}([E'])\cong X^c$, where $X^c$ is the curve obtained by contracting unstable components of $X$. 

By \S \ref{admissible_classification}, $f$ is a cover of type 2C. Thus, $Y=Y_0\cup Y_1$ with $Y_i$ smooth of genus $i$ and $Y_0\cap Y_1=y$, and both marked branch points $y_1,y_2$ lie on $Y_0$. Moreover, $X$ consists of two smooth genus 1 components $X_1,X'_1$ covering $Y_1$, a smooth rational component connecting $X_1$ and $X'_1$ and covering $Y_0$ via a degree 2 map branched over $y_1,y_2$, and $d-2$ other rational components mapping isomorphically to $Y_0$ and attached to $X_1$ or $X'_1$ at the preimages of $y$. Then, $X^c$ is simply the union $X_1\cup X'_1$, where the two components intersect transversely at $x$. The isomorphism $t_{E}([E'])\cong X^c$ is thus an isomorphism $\phi:E\cup E'\cong X_1\cup X'_1$. 

An automorphism of such an object would then consist of an automorphism of $E'$ and an automorphism $f$ compatible with $\phi$. Note that we may not apply a non-trivial automorphism to $X_1$ because $\phi(E)=X_1$, and thus we may not apply a non-trivial automorphism to $Y_1$. A non-trivial automorphism of $X'_1$ over $Y_1$ has no fixed points, so cannot be compatible with an automorphism of $E'$. Therefore, $(E',f,\phi)$ has no non-trivial automorphisms.
\end{proof}

\begin{lem}\label{adm_delta01_mult2}
At every point of intersection of $\pi_{2/1,d}$ and $t_E$, the intersection multipliticity is 2.
\end{lem}

\begin{proof}
Suppose we are given such a point, as described in the proof of Lemma \ref{groupoid_int_with_delta01}. By Proposition \ref{adm_defos}, $\Adm_{2/1,d}$ is smooth at $[f:X\to Y]$, and the deformation theory of $f$ is equivalent to that of $(Y,y_1,y_2)$, where $y_1,y_2$ are the branch points of $f$. The two directions for first-order deformations of $(Y,y_1,y_2)$ vary the elliptic curve $Y_1$ and smooth the node at $y$. The former maps to a first order deformation on $\overline{\cM_2}$ varying the elliptic curves $X_1,X'_1$ simultaneously. A deformation of $Y$ smoothing of the node at $y$ simultaneously smooths all of the nodes of $X$ above $y$. Contracting the rational bridge between $X_1$ and $X'_1$ in this deformation introduces a rational double point in the total space of the induced deformation of $X$, so the node of $X^c=X_1\cup X'_1$ is smoothed to order 2.

Let $\bC[[x,y,z]]$ be the complete local ring of $\overline{\cM_2}$ at $[X^c]$, so that the quotient $\bC[[x,y,z]]\to\bC[[z]]$ is induced by $t_E:\overline{\cM_{1,1}}\to\overline{\cM_2}$. Let $\bC[[s,t]]$ be the complete local ring of $\Adm_{2/1,d}$ at $[f]$. By the above, we may choose the coordinates $s,t$ such that:
\begin{itemize}
\item the quotient map $\bC[[s,t]]\to\bC[[s]]$ corresponds to the deformation varying $X_1,X'_1$ simultaneously according to the versal deformation of $Y_1$,
\item $t$ is a smoothing parameter for the node $X_1\cap X'_1$, and
\item the induced morphism on complete local rings $\bC[[x,y,z]]\to\bC[[s,t]]$ to send $x\mapsto s,y\mapsto t^2$. 
\end{itemize}
The complete local rings at points of intersection of $\pi_{2/1,d}$ and $t_E$ are thus isomorphic to $\bC[t]/t^2$, so the conclusion follows.
\end{proof}

\begin{cor}\label{adm_intersect_delta01_number}
We have 
\begin{equation*}
\pi_{2/1,d}\cdot\delta_{01}=\left[\left(\frac{1}{12}-\frac{d}{2}\right)\sigma_1(d)+\frac{5}{12}\sigma_{3}(d)\right]p.
\end{equation*}
in $A^{*}(\overline{\cM_2})$.
\end{cor}

\begin{proof}
By Lemma \ref{compare_delta01_to_map_from_M11}, the intersection in question is half of $[\pi_{2/1,d}]\cdot [t_E]$. By Lemmas \ref{groupoid_int_with_delta01} and \ref{adm_delta01_mult2}, $[\pi_{2/1,d}]\cdot [t_E]$ consists of a number of degree 2 points equal to the size of $(\overline{\cM_{1,1}}\times_{\overline{\cM_2}}\Adm_{2/1,d})(\Spec(\bC))$. By the description given in Lemma \ref{groupoid_int_with_delta01}, it suffices to enumerate the number of (isomorphism classes of)  \textit{ordered} pairs of isogenies $E\to E''$ and $E'\to E''$ whose degrees sum to $d$. (Note that in this case swapping the labels on the two marked branch points results in an isomorphic cover.) By Lemma \ref{count_isogenies}, this number is $\sum_{i=1}^{d-1}\sigma_1(i)\sigma_1(d-i)$. Finally, by \cite[Table IV, 1]{ramanujan}, we have
\begin{equation*}
\sum_{i=1}^{d-1}\sigma_1(i)\sigma_1(d-i)=\left(\frac{1}{12}-\frac{d}{2}\right)\sigma_1(d)+\frac{5}{12}\sigma_{3}(d),
\end{equation*}
so we are done.
\end{proof}

\section{Intersection with the second test family}\label{eta_E_intersection}

Throughout this section, fix a general elliptic curve $(E,0)$. 

\subsection{The family $\eta_E$}

Consider the projection to the first factor $p:E\times E\to E$, and let $b:Y\to E\times E$ be the blowup of $E\times E$ at the origin. Let $f:Y\to E$ be the composition $p\circ b$. Let $A$ and $B$ be the proper transforms of $E\times0$ and $\Delta$, respectively, so that $A$ and $B$ are disjoint in $Y$. Then, let $X$ be the irreducible surface obtained from $Y$ by gluing $A$ to $B$ along the fibers of $f$, so that we have a factorization $f:Y\to X\to E$. The first map $\nu:Y\to X$ normalizes the general fiber of $g:X\to E$. The fiber $g^{-1}(x)$ is the irreducible nodal curve formed by gluing $0$ to $x$ in $E$ if $x\neq0$, and $g^{-1}(0)$ is the genus 2 curve formed by attaching a nodal cubic to $E$ at $0$.

The morphism $g:X\to E$ is a family of genus 2 curves, and we denote by $\eta_E:E\to\overline{\cM_2}$ the corresponding morphism. We compute the corresponding bivariant class $[\eta_E]\in A^2(\overline{\mathcal{M}_2})$ by computing its intersection with the divisor classes $\delta_{1}$ and $\lambda_1$.

\begin{lem}\label{hodge_degree0}
We have $c_1(\Lambda^2g_{*}\omega_{X/E})=0$.
\end{lem}

\begin{proof}

As $g_{*}\omega_{X/E}$ is a rank 2 vector bundle on the curve $E$, we have
\begin{equation*}
c_1(\Lambda^2g_{*}\omega_{X/E})=c_1(g_{*}\omega_{X/E}).
\end{equation*}

We also have an injection $\omega_{X/E}\to\nu_{*}\omega_{Y/E}(A+B)$. After pushing forward to $E$, we obtain an injection of rank 2 vector bundles $g_{*}\omega_{X/E}\to f_{*}\omega_{Y/E}(A+B)$, which is an isomorphism on each fiber. It follows that $g_{*}\omega_{X/E}\cong f_{*}\omega_{Y/E}(A+B)$.

Thus, it is left to compute $c_1(f_{*}\omega_{Y/E}(A+B))$. First, note that 
\begin{equation*}
\omega_{Y/E}\cong f^{!}\mathcal{O}_E\cong f^{!}\omega_{E}\cong\omega_{Y}.
\end{equation*}
By Grothendieck-Riemann-Roch,
\begin{equation}\label{grr}
\ch(f_{*}\omega_Y(A+B))\td(T_E)=f_{*}(\ch(\omega_Y(A+B))\cdot\td(T_Y)).
\end{equation}

Because $E$ is an elliptic curve, $\td(T_E)=1$. Moreover, $\omega_Y(A+B)$ has vanishing $H^1$ on every fiber of $f$, so all higher pushforwards of $\omega_Y(A+B)$ vanish. Hence the class $f_{*}\omega_Y(A+B)\in K^0(E)$ is simply the sheaf $f_{*}\omega_Y(A+B)$, and the left hand side of (\ref{grr}) is
\begin{equation*}
\ch(f_{*}\omega_Y(A+B))=2+c_1(f_{*}\omega_Y(A+B)).
\end{equation*}

We now compute the right hand side of (\ref{grr}). Let $K=\omega_Y$ be the class of the canonical bundle in $A^1(Y)$. Because the canonical bundle on $E\times E$ is trivial, $K$ is exactly the class of the exceptional divisor of the blowup $b$. Thus, 
\begin{align*}
\ch(\omega_Y(A+B))&=\ch(K+A+B)\\
&=1+(K+A+B)+\frac{1}{2}(K^2+A^2+B^2+2K\cdot A+2K\cdot B+2A\cdot B)\\
&=1+(K+A+B)+\frac{1}{2}(-x-x-x+2x+2x+0)\\
&=1+(K+A+B)+\frac{1}{2}x,
\end{align*}
where $x$ denotes the class of a point on the exceptional divisor of $b$.

Next, to compute $\td(T_Y)$, we apply \cite[Theorem 15.4]{fulton}. Letting $j:\mathbb{P}^1\to Y$ be the inclusion of the exceptional divisor on $Y$, we get:

\begin{equation*}
c(T_Y)=b^{*}c(T_{E\times E})+j_{*}(\alpha),
\end{equation*}
where $\alpha=\frac{1}{x}(1-(1+x)(1-x)^2)=\frac{1}{x}(-x+x^2+x^3)=-1+x$, where $x$ is as above. Because $T_{E\times E}$ is trivial, we are left with

\begin{equation*}
c(T_Y)=1-K+x,
\end{equation*}
from which it follows that

\begin{equation*}
\td(T_Y)=1+\frac{1}{2}c_1+\frac{1}{12}(c_1^2+c_2)=1-\frac{1}{2}K.
\end{equation*}

Putting everything together, we find

\begin{align*}
c_1(f_{*}\omega_Y(A+B))&=f_{*}(\ch(\omega_Y(A+B))\cdot\td(T_Y))-2\\
&=f_{*}\left(\left(1+(K+A+B)+\frac{1}{2}x\right)\left(1-\frac{1}{2}K\right)\right)-2.\\
&=f_{*}\left(A+B+\frac{1}{2}x-\frac{1}{2}K(K+A+B)\right)-2\\
&=f_{*}(A+B)-2\\
&=0.
\end{align*}
\end{proof}

\begin{lem}\label{etaE_and_delta1_transverse}
The intersection of the morphisms $\eta_E:E\to\overline{\mathcal{M}_{2}}$ and $b_{1}:\overline{\cM_{1,1}}\times\overline{\cM_{1,1}}\to \overline{\mathcal{M}_{2}}$ is the disjoint union of two reduced points of degree 1.
\end{lem}

\begin{proof}
The only reducible fiber of $g:X\to E$ is $g^{-1}(0)$. Thus, the objects of the groupoid $E\times_{\overline{\mathcal{M}_2}}(\overline{\cM_{1,1}}\times\overline{\cM_{1,1}})(\Spec(\bC))$ are the triples consisting of the point $0\in E$, a pair of elliptic curves $([E_1],[E_2])\in\overline{\cM_{1,1}}\times\overline{\cM_{1,1}}$, and an isomorphism $E_{1}\cup_{0}E_2\cong g^{-1}(0)$. Up to isomorphism, there are only two such objects, which correspond to the choices of ordering of the two components of $g^{-1}(0)$, and it is clear that these objects have no non-trivial automorphisms.

Thus, it suffices to show that $\eta_E$ and $b_1$ intersect transversely, i.e., no non-trivial tangent vectors $v$ to $[g^{-1}(0)]\in\overline{\cM_2}$ factor through both $b_1$ and $\eta_E$. If $v$ factors through $\eta_E$, then the corresponding first-order deformation of $X_0$ smooths the separating node $x\in X_0$, as $X$ is smooth at $x$. On the other hand, first-order deformations factoring through $b_1$ do not smooth the node at $x$, so we must have $v=0$. As $b_{1},\eta_E$ are unramified over $[g^{-1}(0)]$, we conclude that $\eta_E,b_1$ intersect transversely.
\end{proof}

\begin{cor}\label{class_of_etaE}
We have $[\eta_E]=8\delta_{00}$.
\end{cor}
\begin{proof}
By Lemma \ref{hodge_degree0}, we have $[\eta_E]\cdot\lambda_1=0$. By Lemma \ref{etaE_and_delta1_transverse}, we have $[\eta_E]\cdot\delta_1=p$, as $b_1$ has degree 2. The result follows from Theorem \ref{M2_relations}.
\end{proof}

\subsection{Intersection of $\pi_{2/1,d}$ and $\eta_E$}

By \S \ref{admissible_classification}, the intersection of $\pi_{2/1.d}$ with $\eta_E$ consists of covers of types 2B, 3C, and 4C. However, if $E$ is general, the source of a cover of type 4C cannot contract to a member of $\eta_E$. Thus, we only consider covers of types 2B and 3C. The intersection in question is the stack $E\times_{\overline{\cM_2}}\Adm_{2/1,d}$, whose $\Spec(\bC)$-points consist of the data of a point $0\neq x$, an admissible cover $f:X\to Y$ of type 2B or 3C, and an isomorphism $h:X^c\cong E/(0\sim x)$.

\subsubsection{Covers of type 2B}

We consider the contribution to $E\times_{\overline{\cM_2}}\Adm_{2/1,d}$ coming from covers of type 2B. We have $f:X\to Y$, $Y=Y_0\cup Y_1$ where $Y_i$ has genus $i$, $y=Y_0\cap Y_i$, and both marked points $y_1,y_2$ lie on $Y_0$. $X$ contains a component $X_1$ of genus 1 mapping to $Y_1$, and $d-1$ rational components, all of which map isomorphically to $Y_0$ except for one that connects two points of $X_1$ lying over $y$ and has degree 2 over $Y_0$. Then, $X^c$ is the irreducible curve obtained by gluing these two points together on $X_1$.

\begin{lem}\label{adm_eta_2b_set}
The points of $(E\times_{\overline{\cM_2}}\Adm_{2/1,d})(\Spec(\bC))$ where the cover $f$ has type 2B have no non-trivial automorphisms.
\end{lem}
\begin{proof}
Clear.
\end{proof}

\begin{lem}\label{adm_eta_2b_mult2}
At every point of intersection of $\pi_{2/1,d}$ and $\eta_E$ where the cover $f$ has type 2B, the intersection multipliticity is 2.
\end{lem}
\begin{proof}
The proof is similar to the proof of Lemma \ref{adm_delta01_mult2}. The deformation directions for the cover $X\to Y$ vary the elliptic component of $Y$ or smooth the node. These induce deformations of $X^c$ which vary the elliptic curve $X_1$ (non-trivially, to first order) and smooth the node of $X^c$ to order 2, respectively. The tangent direction coming from $E$ moves apart the glued points of $X$, which is orthogonal to the other two directions. We conclude in a similar way to Lemma \ref{adm_delta01_mult2} that the intersection multiplicities are all 2.
\end{proof}

\begin{cor}\label{adm_eta_2b_total}
The contribution of $[\pi_{2/1,d}]\cdot [\eta_E]$ from covers of type 2B is $2\sigma(d)(d-1)p$.
\end{cor}
\begin{proof}
By Lemmas \ref{adm_eta_2b_set} and \ref{adm_eta_2b_mult2}, it suffices to show that the subset of $(E\times_{\overline{\cM_2}}\Adm_{2/1,d})(\Spec(\bC))$ corresponding to covers of type 2C has cardinality $\sigma(d)(d-1)$. 

To do this, we may give a bijection between this subset and the set of pairs $(G,g)$ where $G\subset E[d]$ is a subgroup of order $d$ and $0\neq g\in G$. The claim then follows from Lemma \ref{count_isogenies}. In one direction, given a triple $(x,f,h)$, let $\widetilde{h}:X_1\to E$ be the induced isomorphism after normalizing both sides. Then, we send $(x,f,h)$ to $G=\mathrm{ker}(f|_{X_1}\circ\widetilde{h}^{-1})$, where the origin of $Y_1$ is taken to be the node $y$, and $g=\widetilde{h}(x)$. In the other, let $f$ be the map obtained by gluing rational curves in the obvious way to both sides of $E\to E/G$, set $x=g$, and let $h$ be induced by the identity on $E$.
\end{proof}

\subsubsection{Covers of type 3C}

We now turn to admissible covers $f:X\to Y$ of type 3C. Here, $Y$ is a nodal cubic with node $y$ and marked points $y_1,y_2\in Y$, and $X$ is a union of smooth curves $X_1\cong E$ and $X_2,\ldots,X_n\cong\bP^1$, which are arranged in an $n$-gon, or, when $n=1$, $X$ is irreducible of genus 2. Each of $X_2,\ldots,X_n$ maps via $x\mapsto x^a$ to $\bP^1$, then to $Y$ via a normalization, so that $f$ is ramified over $Y$ at all of the nodes of $X$ with ramification index $a\ge2$ at every branch. Then, $f$ has degree $a$ when restricted to $X$ as well, and in addition is simply ramified at $x_1,x_2\in X_1$ over $y_1,y_2$, respectively. The total degree of $f$ is therefore $d=an$.

\begin{lem}\label{adm_eta_3c_set}
The points of $(E\times_{\overline{\cM_2}}\Adm_{2/1,d})(\Spec(\bC))$ where $f$ has type 3C have automorphism group of order $a^{n-1}$.
\end{lem}
\begin{proof}
The data of an automorphism of such an object is that of an automorphism of $f$ that induces the identity on $X^c$, and thus on $X_1$. Thus, we may only apply automorphisms of the components $X_2,\ldots,X_n$ over $Y$. Each component has $a$ automorphisms over $Y$, corresponding to multiplying by $a$-th roots of unity, so we have $a^{n-1}$ in total.
\end{proof}

\begin{lem}\label{adm_eta_3c_mult2}
At every point of intersection of $\pi_{2/1,d}$ and $\eta_E$ where the cover $f$ has type 3C, the (non-stacky, on the level of complete local rings) intersection multipliticity is $na^{n-1}$.
\end{lem}
\begin{proof}
By Proposition \ref{adm_defos}, the complete local ring at $f$ is
\begin{equation*}
\bC[[s,t,t_1,\ldots,t_n]]/(t=t_1^a=\cdots=t_n^a),
\end{equation*}
which is canonically a $\bC[[s,t]]$-algebra via the forgetful map $\psi_{2/1,d}:\Adm_{2/1,d}\to\overline{\cM_{1,2}}$. Here, $t$ is a smoothing parameter for the node of $Y$ and $t_i$ is a smoothing parameter for the node $X_i\cap X_{i+1}$ (where we denote $X_{n+1}=X_1$). The parameter $s$ corresponds to the deformation of the base that moves $y_1$ and $y_2$ apart. More precisely, the quotient map $\bC[[s,t]]\to\bC[[s]]$ is induced by the morphism $M_{0,4}\to\overline{\cM_{1,2}}$ sending $(\bP^1,0,1,\infty,\lambda)$ to the marked nodal cubic $(\bP^1/(0\sim 1),\infty,\lambda)$.

Let $\bC[[x,y,z]]$ be the complete local ring of $\overline{\cM_2}$ at $[X^c]$, where the coordinates are chosen as follows. The coordinate $z$ is the deformation parameter for a deformation of $X^c$ that moves apart the nodes of $X_1$ (more precisely, the quotient map $\bC[[x,y,z]]\to\bC[[z]]$ is induced by $\eta_E$). The coordinate $x$ is the deformation parameter of $X^c$ that varies the elliptic curve $(X_1,x_1)$. Finally, the coordinate $y$ is a smoothing parameter for the node of $X^c$.

The morphism $\pi_{2/1,d}$ thus induces a map on complete local rings
\begin{equation*}
\rho:\bC[[x,y,z]]\to \bC[[s,t,t_1,\ldots,t_n]]/(t=t_1^a=\cdots=t_n^a),
\end{equation*}

We claim that, possibly after harmless renormalizations of the coordinates above, we have 
\begin{align*}
\rho(x)&=s+v_x\\
\rho(y)&= t_1\cdots t_nu_y,
\end{align*}
where $v_x\in(t_1,\ldots,t_n)$ and $u_y$ is a unit.

From here, we will conclude that the complete local ring of $E\times_{\overline{\cM_2}}\Adm_{2/1,d}$ is of the form
\begin{equation*}
\bC[[t_1,\ldots,t_n]]/(t_1^a-t_i^a,t_1\cdots t_n),
\end{equation*}
which is a complete intersection of local hypersurfaces of degrees $a,\ldots,a,n$, and the conclusion follows.

The claim above follows from the following three statements:
\begin{enumerate}
\item[(a)] $\rho(x)=s$ after killing all of the $t_i$.
\item[(b)] $\rho(y)=0$ after killing any of the $t_i$.
\item[(c)] $\rho(y)=ct_1^n+O(t_1^{n+1})$ after setting all of the $t_i$ equal to $t_1$ and killing $s$, where $c\in\bC^{*}$.
\end{enumerate}

First, consider (a). Killing all of the $t_i$ corresponds to taking a 1-parameter deformation of $X^c$ that results from applying the deformation of $(Y,y_1,y_2)$ corresponding to $\bC[[s,t]]\to\bC[[s]]$. Let $\widetilde{Y}\cong\bP^1$ be the normalization of $Y$, so that $X_1$ is a 4-point cover of $\widetilde{Y}$, ramified over $0,1,y_1=\infty,y_2=\lambda$. It suffices to show that varying $\lambda$ to first order varies the elliptic curve $(Y,f^{-1}(\infty))$ to first order as well. Indeed, after specifying the appropriate monodromy above $0,1,\infty,\lambda$, we get a dominant and thus generically unramified map of stacky curves $M_{0,4}\to\cM_{1,1}$ sending $(0,1,\lambda,\infty)$ to the unique elliptic curve branched over $\widetilde{Y}$ in the specified way. As $E$ is chosen to be general, the claim follows, and after renormalizing the coordinate $s$ we may conclude $\rho(x)=s\pmod{(t_1,\ldots,t_n)}$.

Statement (b) is immediate from the fact that killing any $t_i$ results in a deformation of $X^c$ that is nodal on the general fiber.

Finally, statement (c) says that if we smooth the nodes of $X$ simultaneously and do not vary the component $X_1$, contracting the rational components on the special fiber yields a smoothing of order $n$ at the contracted point. Indeed, the contraction map is resolution of an $A_n$-surface singularity in the total space of the deformation of $X$. This completes the proof.
\end{proof}

\begin{cor}\label{adm_eta_3c_total}
The contribution of $[\pi_{2/1,d}]\cdot [\eta_E]$ from covers of type 3C is $2\sigma(d)(d-1)p$.
\end{cor}
\begin{proof}
By Lemmas \ref{adm_eta_3c_set} and \ref{adm_eta_3c_mult2}, we get a contribution of $n$ in this intersection for each admissible cover of type 3C where $X$ has $n$ components and has $X_1\cong E$. By Lemma \ref{count_covers_from_elliptic_curve}, there is one such cover for each nonzero $(d/n)$-torsion point of $E$, disregarding the marked points on the base. When $a=d/n=2$, these covers are all the same, but the non-zero torsion points give rise to a different admissible covers, distinguished by where the rational components are attached on $X_1$. Because we have two choices for the markings of the simple branch points of $f$, the total contribution to $[\pi_{2/1,d}]\cdot [\eta_E]$ from covers of type 3C is
\begin{align*}
2\left(\sum_{an=d}(a^2-1)n\right)p=2\left(\sum_{a|d}\left(ad-\frac{d}{a}\right)\right)p=2\sigma(d)(d-1)p.
\end{align*}
\end{proof}

\section{The class of $\pi_{2/1,d}$}\label{main_pf}

We now combine the results of the previous two sections to complete the main calculation.

\begin{proof}[Proof of Theorem \ref{main_thm}]
Let $[\pi_{2/1,d}]=x\delta_0+y\delta_1$. By Corollary \ref{adm_intersect_delta01_number} and Theorem \ref{M2_relations}, we have, from intersecting $\pi_{2/1,d}$ with $t_E$,
\begin{equation*}
\frac{1}{4}x-\frac{1}{48}y=\left(\frac{1}{12}-\frac{d}{2}\right)\sigma_1(d)+\frac{5}{12}\sigma_{3}(d).
\end{equation*}
Combining Corollaries \ref{class_of_etaE}, \ref{adm_eta_2b_total}, and \ref{adm_eta_3c_total} with Theorem \ref{M2_relations}, we also have, from intersecting $\pi_{2/1,d}$ with $\eta_E$,
\begin{equation*}
-2x+y=4\sigma_1(d)(d-1).
 \end{equation*}
Solving for $x$ and $y$, we obtain
\begin{equation*}         
[\pi_{2/1,d}]=\left(2\sigma_{3}(d)-2d\sigma_{1}(d)\right)\delta_0+\left(4\sigma_{3}(d)-4\sigma_{1}(d)\right)\delta_1,
\end{equation*}
as desired.
\end{proof}

We now prove Corollary \ref{modularity_statement}, that the coefficients of $\delta_0$ and $\delta_1$, when assembled into generating functions, are quasi-modular forms.

\begin{proof}[Proof of Corollary \ref{modularity_statement}]
Recall that $\Qmod$ is generated over $\bQ$ by the Eisenstein series
\begin{equation*}
E_{2n}(q)=1+\frac{(2\pi i)^{2n}}{(2n-1)!\cdot\zeta(2n)}\sum_{k=1}^{\infty}\frac{k^{2n-1}q^{2k}}{1-q^{2k}}.
\end{equation*}
Moreover, we have the standard formulas
\begin{align*}
\sum_{d=1}^{\infty}\sigma_1(d)q^k&=\sum_{k=1}^{\infty}\frac{kq^k}{1-q^k},\\
\sum_{d=1}^{\infty}\sigma_3(d)q^k&=\sum_{k=1}^{\infty}\frac{k^3q^k}{1-q^k}
\end{align*}
We also have the Ramanujan differential equation
\begin{equation*}
q\frac{dP}{dq}=\frac{P^2-Q}{12},
\end{equation*}
where
\begin{align*}
P&=1-24\sum_{k=1}^{\infty}\frac{kq^k}{1-q^k}\\
Q&=1+240\sum_{k=1}^{\infty}\frac{k^3q^k}{1-q^k},
\end{align*}
see \cite{by}.
Therefore, we may write
\begin{equation*}
\sum_{d}[\pi_{2/1,d}]q^d=T_0\delta_0+T_1\delta_1,
\end{equation*}
where
\begin{equation*}
T_0=\frac{E_4-E_2^2}{6}+\frac{E_4-1}{120}
\end{equation*}
and
\begin{equation*}
T_1=\frac{E_4-1}{60}+\frac{1-E_2}{6}.
\end{equation*}
The conclusion is immediate.
\end{proof}

We end this section by speculating about possible generalizations of these results to higher genus.

\begin{question}
For any $g\ge 2$, let $\pi_{g/1,d}:\Adm_{g/1,d}\to\overline{\cM_g}$ be the map remembering the stabilized source of an admissible cover of a genus 1 curve by a genus $g$ curve.
\begin{enumerate}
\item[(a)] Is $[\pi_{g/1,d}]\in A^{*}(\cM_g)$ tautological?
\item[(b)] Let $\alpha\in A^{*}(\cM_g)$ be a fixed class of complimentary dimension to $\pi_{g/1,d}$. Is
\begin{equation*}
\sum_{d\ge2}\left(\int_{\overline{\cM_g}}[\pi_{g/1,d}]\cdot\alpha\right) q^d\in\Qmod\text{ ?}
\end{equation*}
\end{enumerate}
\end{question}

In light of the examples of \cite{gp} and \cite{vanzelm}, it seems natural to expect that the answer to (a) is no when $d$ is fixed and $g$ is sufficiently large compared to $d$. Analogous quasi-modularity results to (b) hold, in the setting of Gromov-Witten theory, for covers of \textit{fixed} elliptic curves, see \cite[\S 5]{op}. Our method of computation extends, in principle, to higher genus, provided one can classify admissible covers and produce enough test classes, but this approach is clearly impractical without new ideas. We point out that Schmitt-van Zelm \cite{schmittvanzelm} have developed an algorithm for intersecting tautological classes on $\overline{\cM_g}$ with loci of \textit{Galois} admissible covers, whose deformation spaces are smooth.

\section{Proof of Theorem \ref{enum_thm}}\label{application}

In this section, we give an application of Theorem \ref{main_thm}. Given five very general points $x_1,\ldots,x_5\in\bP^1$, we compute the number of points $x_6$ such that the hyperelliptic curve branched over the $x_i$ is $d$-elliptic. We may associate to $x_1,\ldots,x_5$ a 1-parameter family $\mu_P$, and compute its intersection with $\mu_{2/1,d}$; after exhibiting enough control on the covers that appear in this intersection, we deduce Theorem \ref{enum_thm}.

\subsection{The Class of $\mu_P$}

Let $P\in\bC[x]$ be a square-free monic polynomial of degree 5. We define a map $\mu_P:\bP^1\to\overline{\cM_2}$ sending $t$ to the hyperelliptic curve branched over $t^2$ and the roots of $P$, and denote the corresponding class by $[\mu_P]\in A^2(\overline{\cM_2})$. More precisely, the total space $X$ has charts
\begin{align*}
U_1&=\Spec\bC[x,y,t]/(y^2-P(x)(x-t^2))\\
U_2&=\Spec\bC[x,y',x]/({y'}^{2}-P(x)(s^2x-1))\\
U_3&=\Spec\bC[u,v,t]/(v^2-u^5P(u^{-1})(1-ut^2))\\
U_4&=\Spec\bC[u,v',s]/({v'}^2-u^5P(u^{-1})(s^2-u)).
\end{align*}

The transition functions are as follows: between $U_1$ and $U_2$, we have $t=1/s$ and $y=ity'$ (where $i$ is a square root of $-1$), between $U_1$ and $U_3$, we have $u=1/x$ and $v=u^3y$, and between $U_2$ and $U_4$, we have $u=1/x$ and $v'=u^3y'$. Then, we have a family $g:X\to\bP^1=\Proj\bC[s,t]$ of stable genus 2 curves.

\begin{prop}\label{class_of_mu}
We have $[\mu_P]=16\delta_{00}+96\delta_{01}$ in $A^2(\overline{\cM_2})$.
\end{prop}

\begin{proof}
First, observe that $[\mu_P]\cdot\delta_1=0$; indeed, every member of the family is irreducible. Next, we compute $[\mu_P]\cdot\delta_0$, by intersecting $\mu_P$ with the morphism $b_0:\overline{\cM_{1,2}}\to\overline{\cM_2}$; we have $[b_0]=2\delta_0$. Set-theoretically, there are 10 nodal fibers of the family $\mu_P$, corresponding to the points where $t^2$ becomes equal to one of the 5 roots of $P$. Thus, there are 20 $\bC$-points in the intersection of $\mu_P$ and $b_0$, which clearly have no automorphisms. The total space $X$ of the family $\mu_P$ has an ordinary double point at each node, so each point has intersection multiplicity 2. Thus, $[\mu_P]\cdot\delta_0=20p$, and the formula for $[\mu_P]$ follows from Theorem \ref{chow_ring_m2}.
\end{proof}

%
%

\subsection{Intersection with $\pi_{2/1,d}$}\label{control_intersection}

\begin{lem}\label{mu_and_pi_type1}
Suppose $P$ is a general monic square-free polynomial of degree 5. Then, $\mu_P$ and $\pi_{2/1,d}$ intersect in the dense open substack $\Adm_{2/1,d}^{\circ}\subset\Adm_{2/1,d}$ of covers of type 1, i.e., covers $f:X\to Y$ with $X$ and $Y$ smooth.
\end{lem}

\begin{proof}
We the family $\mu_P$ to avoid two loci: smooth curves arising from covers of type 2A/2A' as contractions of their sources, and irreducible nodal curves arising as elliptic curves with a $d$-torsion point glued to its origin. Let $Z\subset\overline{\cM_2}$ be the union of these loci; note that $Z$ is 1-dimensional. Given a point $[C]\in\overline{\cM_2}$ in the image of some $\mu_P$, there is a 3-dimensional family of $P$ such that $\mu_P$ passes through $[C]$, and a 5-dimensional space of $P$, so $P$ may be chosen so that $\mu_P$ avoids $Z$.
\end{proof}

\begin{lem}\label{control_aut_groups}
Suppose $P$ is a general monic square-free polynomial of degree 5, and $s\in\bC$ is not a root of $P$. Then, the genus 2 curve $C$ associated to the affine equation $y^2=f(x)(x-s)$ has either $\#\mathrm{Aut}(C)=2$ or $\#\mathrm{Aut}(C)=4$. The latter occurs if and only if $C$ is bielliptic.
\end{lem}

\begin{proof}
The assertion is equivalent to the following: a general choice of distinct points $x_1,\ldots,x_5\in\bP^1$ has the property that for all $x_6\in\bP^1-\{x_1,\ldots,x_5\}$, the group of automorphisms of $\bP^1$ fixing the set $\{x_1,\ldots,x_6\}$ has order at most 2, and the order is 2 if and only the hyperelliptic curve branched over $x_1,\ldots,x_6$ is bielliptic.

One checks that for $x_1,\ldots,x_5$ general, there is no automorphism of $\bP^1$ doing any of the following:
\begin{enumerate}
\item $x_1\mapsto x_2\mapsto x_3\mapsto x_4\mapsto x_5$
\item $x_1\mapsto x_2\mapsto x_3\mapsto x_4\mapsto x_1$
\item $x_1\mapsto x_2\mapsto x_3$ and $x_4\mapsto x_5\mapsto x_4$
\item $x_1\mapsto x_2\mapsto x_3\mapsto x_1$ and $x_4\mapsto x_5$
\item $x_1\mapsto x_1$ and $x_2\mapsto x_3\mapsto x_4\mapsto x_2$
\item $x_1\mapsto x_1$ and $x_2\mapsto x_3\mapsto x_4\mapsto x_5$
\item $x_1\mapsto x_1$, $x_2\mapsto x_3\mapsto x_2$ and $x_4\mapsto x_5$
\item $x_1\mapsto x_1$, $x_2\mapsto x_2$, and $x_3\mapsto x_4\mapsto x_3$.
\item $x_1\mapsto x_1$, $x_2\mapsto x_2$, and $x_3\mapsto x_4\mapsto x_5$.
\end{enumerate}
For instance, to prove that $x_1\mapsto x_2\mapsto x_3\mapsto x_1$ and $x_4\mapsto x_5$ is impossible, we may let $(x_1,x_2,x_3)=(0,1,\infty)$, and note that the unique automorphism of $\bP^1$ sending $x_1\mapsto x_2\mapsto x_3\mapsto x_1$ will only send $x_4\mapsto x_5$ if the latter two points are chosen in special position.

We then conclude that a non-trivial permutation $\rho$ of the $x_i$ must be a union of three 2-cycles. An automorphism of $\bP^1$ inducing such a permutation must be an involution. The quotient of the induced involution on $C$ will have genus 1, as the involution is not hyperelliptic.
\end{proof}

\begin{prop}\label{mu_and_pi_intersection_nice}
For a very general monic square-free polynomial $P$ of degree 5, every cover $f:C\to E$ in the intersection of $\mu_P$ and $\pi_{2/1,d}$ has the following properties:
\begin{enumerate}
\item $f$ is an admissible cover of type 1, i.e., $C$ and $E$ are smooth.
\item $J(C)$ is not isogenous to the product of an elliptic curve with itself.
\item Given an automorphism $h$ of $C$, there exists an automorphism $h'$ of $E$ and a cover $f':C\to E$ compatible with $h$ and $h'$.
\item The intersection of $\mu_P$ and $\pi_{2/1,d}$ is transverse (of the expected dimension) at $[f]$.
\end{enumerate}
\end{prop}

\begin{proof}
The first condition is Lemma \ref{mu_and_pi_type1}.

The locus of smooth curves in $\overline{\cM_2}$ whose Jacobian is isogenous to a self-product of an elliptic curve is a countable union of substacks of dimension 1, because we have a 1-dimensional family of elliptic curves, and each elliptic curve has a countable set of isogeny quotients. By the same argument as in Lemma \ref{mu_and_pi_type1}, a very general $P$ avoids this locus.

By Lemma \ref{control_aut_groups}, the only possible automorphisms of $C$ are the hyperelliptic involution, which is compatible with an involution of $E$, or a bielliptic involution, which commutes with $f$, by Corollary \ref{complement_not_isogenous} and condition (ii).

Finally, we check the transversality. Consider the diagram
\begin{equation*}
\xymatrix{
W \ar[r] \ar[d] & \Adm_{2/1,d}^{\circ} \ar[d]\\
\bA^{6}-\Delta \ar[r] \ar[d] & \cM_2 \\
\bA^5-\Delta
}
\end{equation*}
where the map $\bA^6-\Delta\to\cM_2$ is the smooth morphism taking six distinct points $x_1\ldots,x_6$ to the hyperelliptic curve branched over $x_1,\ldots,x_6$, the map to $\bA^5-\Delta$ remembers the first five points, and the top square is Cartesian. By Proposition \ref{adm_defos}, $\Adm_{2/1,d}^{\circ}$ is smooth, so $W$ is smooth of dimension 0. Thus, by generic smoothness, the general fiber of $q:W\to\bA^5-\Delta$ is smooth. The previous conditions guarantee that the very general fiber of $q$ is precisely the intersection of $\mu_P$ with $\pi_{2/1,d}$, where $P(x)=(x-x_1)\cdots(x-x_5)$ (we need to assume in addition that $x_6\neq0$, because in the definition of $\mu_P$, we square the varying parameter), so we are done.
\end{proof}

\begin{proof}[Proof of Theorem \ref{enum_thm}]
Fix a coordinate on $\bP^1$, and let $P$ be the monic polynomial with roots $x_1,\ldots,x_5$; we assume $0,\infty\neq x_i$. Let $C_{x_6}$ denote the genus 2 curve branched over $x_1,\ldots,x_6$. For a very general collection of $x_1,\ldots,x_5$, the conditions of Proposition \ref{mu_and_pi_intersection_nice} are satisfied. The $\bC$-points of the intersection of $\mu_P$ and $\pi_{2/1,d}$ consist of the data of $t\in\bC$, a cover $f:C\to E$ (with the data of its branch points), and an isomorphism $h:C_{x_6}\cong C$. By condition (iii), we may disregard $h$. It is moreover clear that these objects have no automorphisms. By condition (ii) and Corollary \ref{complement_not_isogenous}, $f:C\to E$ factors uniquely through exactly one of two optimal covers $f_i:C\to E_i$ of some degree $d'|d$.

Let $b_{d'}$ be the number of $x_6$ such that $C_{x_6}$ admits an optimal cover $C_{x_6}\to E$ of degree $d'$. Then, we have
\begin{equation*}
a_d=\sum_{d'|d}b_{d'}=(b\star1)_d,
\end{equation*}
where $\star$ denotes Dirichlet convolution and $1_d=1$. Given a $d$-elliptic curve $C_{x_6}$, the number of points in the intersection of $\mu_{P}$ and $\pi_{2/1,d}$ factoring through an optimal cover from $C$ of degree $d'$ is $4\sigma_1(d/d')$. Indeed, we have 2 choices for the optimal cover, 2 choices for the labelling of the branch points on $E$, and $\sigma_1(d/d')$ quotients of $E$ of degree $d/d'$, by Lemma \ref{count_isogenies}. Because we may assume $x_6\neq0,\infty$, the family $\mu_P$ contains $C_{x_6}$ twice, so we have
\begin{equation*}
[\mu_{P}]\cdot[\pi_{2/1,d}]=8(b\star\sigma)_d\cdot p
\end{equation*}
Applying Theorem \ref{main_thm} and Proposition \ref{class_of_mu}, we have
\begin{equation*}
(b\star\sigma)_d=5(\sigma_{3}(d)-d\sigma_{1}(d)).
\end{equation*}
However, as $\sigma=\Id\star 1$, where $\Id_d=d$, and $a=b\star 1$, we have
\begin{equation*}
(a\star\Id)_d=5(\sigma_{3}(d)-d\sigma_{1}(d)).
\end{equation*}
Because $\Id_{d}^{-1}=\mu(d)d$ is a Dirichlet inverse of $\Id$, we get
\begin{align*}
a_d&=5\sum_{d'|d}(\sigma_{3}(d')-d'\sigma_{1}(d'))\cdot\mu\left(\frac{d}{d'}\right)\cdot\frac{d}{d'}\\
&=5d\left[\sum_{d'|d}\left(\frac{\sigma_{3}(d')}{d'}\cdot\mu\left(\frac{d}{d'}\right)\right)-d\right],
\end{align*}
where we have applied M\"{o}bius Inversion. The proof is complete.
\end{proof}

\end{document}